\documentclass{amsart}[12pt]

\usepackage{amscd,amssymb, amsmath}
\usepackage{comment}

\usepackage{tikz}
\usepackage{color}

\newtheorem{theorem}{Theorem}[section]
\newtheorem{lemma}[theorem]{Lemma}
\newtheorem{proposition}[theorem]{Proposition}
\newtheorem{corollary}[theorem]{Corollary}

\theoremstyle{definition}

\newtheorem{example}[theorem]{Example}

\newtheorem{remark}[theorem]{Remark}

\newcommand{\GL} {\mathrm{GL}}

\def\QQ {{\mathbb Q}}     

\begin{document}

	\title[Tensor products of Lie nilpotent associative algebras]
	{Tensor products of Lie nilpotent associative algebras and applications to codimension sequences}	
\author[Elitza Hristova]
{Elitza Hristova}
\address{Institute of Mathematics and Informatics,
	Bulgarian Academy of Sciences,
	Acad. G. Bonchev Str., Block 8,
	1113 Sofia, Bulgaria}
\email{e.hristova@math.bas.bg}

\thanks{Partially supported by Grant KP-06-N92/2 of the Bulgarian National Science Fund.}



	\begin{abstract}
	Let $G$ and $H$ be Lie nilpotent associative algebras over a field $K$ such that in addition $H$ satisfies the identity $[x_1, x_2] \cdots [x_{2k-1}, x_{2k}]=0$ for some $k \geq 2$. In this paper, extending results of Deryabina and Krasilnikov, we show that the tensor product $G \otimes H$ is again a Lie nilpotent associative algebra. Moreover, we give a lower and an upper bound on the minimal value of $q$ for which $[x_1, \dots, x_{q+1}] = 0$ is an identity for $G\otimes H$. In the case when $H$ satisfies the identities $[x_1, x_2, x_3] = 0$ and $[x_1, x_2][x_3, x_4] = 0$ and $\operatorname{char}K \neq 3$, we determine a better upper bound for $q$, which in many cases is equal to the minimal index of Lie nilpotency for $G\otimes H$. As a corollary, we reprove a result of Drensky saying that any product of Grassmann algebras of the form $E\otimes E_{i_1}\otimes \cdots \otimes E_{i_s}$ or $E_{j_1} \otimes E_{j_2} \otimes \cdots \otimes E_{j_t}$, where $E$ denotes the Grassmann algebra over a countable dimensional vector space and $E_r$ denotes the Grassmann algebra over an $r$-dimensional vector space, satisfies an identity of the form $[x_1, \dots, x_{q+1}] = 0$. We also provide several particular cases in which the minimal value of $q$ can be explicitly computed. As an application, we consider a field of characteristic zero, the variety $\mathfrak{N}_p$ of Lie nilpotent associative algebras of index at most $p$ and the corresponding relatively free algebras of finite rank, $F_n(\mathfrak{N}_p)$. We exhibit many explicit irreducible $S_n$-modules in the $S_n$-module decomposition of the space of proper multilinear polynomials of degree $n$ in $F_n(\mathfrak{N}_p)$ for any $p$. This gives a lower bound for the dimensions of the spaces of multilinear and proper multilinear polynomials of degree $n$ in $F_n(\mathfrak{N}_p)$.
	\end{abstract}

	\maketitle
	\section{Introduction}
	
	Let $K$ be a field and let $K \left \langle X \right\rangle$ denote the free unital associative algebra generated by the countably infinite set $X = \{x_1, \dots, x_n, \dots\}$ over $K$. We say that an associative unital $K$-algebra $A$ satisfies a polynomial identity if there is a polynomial $f(x_1, \dots, x_m) \in K \left\langle X \right\rangle$ such that $f(a_1, \dots, a_m) = 0$ for all $a_1, \dots, a_m \in A$. If $A$ satisfies a polynomial identity then $A$ is called an algebra with polynomial identity, or simply a PI-algebra. We also define left-normed long commutators in $K \left \langle X \right\rangle$ recursively in the following way: $[u_1, u_2] = u_1u_2 - u_2u_1$ and $[u_1, \dots, u_{p+1}] = [[u_1, \dots, u_p], u_{p+1}]$ for any $p \geq 2$ and $u_1, \dots, u_{p+1} \in K \left \langle X \right\rangle$. A unital associative algebra $A$ is called Lie nilpotent of index at most $p$ if $A$ satisfies the identity $[x_1. \dots, x_{p+1}] = 0$. We will also say that $A$ is of index $p$ if $A$ satisfies $[x_1. \dots, x_{p+1}] = 0$ and does not satisfy $[x_1. \dots, x_p] = 0$.

	Let $E$ denote the Grassmann (or exterior) algebra over a countable dimensional $K$-vector space and let $E_r$ denote the Grassmann algebra over an $r$-dimensional $K$-vector space, where $r \geq 2$ and $\operatorname{char} K \neq 2$. It is well-known that both $E$ and $E_r$ are Lie nilpotent algebras and satisfy the identity $[x_1, x_2, x_3] = 0$. In his paper \cite{D}, Drensky shows that moreover any product of Grassmann algebras of the form $E\otimes E_{i_1}\otimes \cdots \otimes E_{i_s}$ or $E_{j_1} \otimes E_{j_2} \otimes \cdots \otimes E_{j_t}$ satisfies an identity of the form $[x_1, \dots, x_{q+1}]=0$ for some integer $q \geq 2$. There are several cases, in which an explicit value of $q$ is known. For example, it is well-known that $E_2 \otimes E_2$ satisfies the identity $[x_1, x_2, x_3, x_4] = 0$ and that $E_2 \otimes E_2 \otimes E_2$ satisfies the identity $[x_1, \dots, x_5] = 0$. Furthermore, in \cite{DeK}, the authors show that the algebras $E \otimes E_{2k}$ and $E \otimes E_{2k+1}$ satisfy the identity $[x_1, \dots, x_{2k+3}]=0$. This is a corollary of a more general result that they prove in the same paper, namely that if $G$ and $H$ are unital associative algebras satisfying $[x_1, x_2, x_3] = 0$ and such that in addition $H$ satisfies $[x_1, x_2]\cdots [x_{2k-1}, x_{2k}]=0$ for some $k \geq 2$, then $G\otimes H$ satisfies $[x_1, \dots, x_{2k+1}]=0$. This result is valid over any field $K$. Notice that the condition on $H$ to satisfy the identity $[x_1, x_2]\cdots [x_{2k-1}, x_{2k}]=0$ is important, since for example the algebra $E\otimes E$ is not Lie nilpotent.   
	
	In the first part of this paper, we extend the above statement. We show that if $G$ and $H$ are Lie nilpotent associative algebras of arbitrary indices such that in addition $H$ satisfies $[x_1, x_2]\cdots [x_{2k-1}, x_{2k}]=0$ for some $k \geq 2$, then the tensor product $G\otimes H$ is again Lie nilpotent. We also give a lower and an upper bound for the index of Lie nilpotency of $G\otimes H$. In the case when $H$ satisfies $[x_1, x_2, x_3]$ and $[x_1, x_2][x_3, x_4]$, we determine a better upper bound on the index of Lie nilpotency for $G\otimes H$, under the requirement that $\operatorname{char}K \neq 3$. These results are given in Theorem \ref{thm_tensorProduct}  and Proposition \ref{prop_addE2}. Since every finitely generated Lie nilpotent associative algebra statisfies an identity of the form $[x_1, x_2]\cdots [x_{2k-1}, x_{2k}]=0$ (see e.g. \cite{J} or \cite{MPR}), an immediate corollary of Theorem \ref{thm_tensorProduct} is that if $G$ is any Lie nilpotent associative algebra and $H$ is finitely generated Lie nilpotent then $G\otimes H$ is again Lie nilpotent.  
	As another application of Theorem \ref{thm_tensorProduct}, we obtain a new proof of Drensky's result on the tensor products of Grassmann algebras.  
	The advantage of this approach is that we are able to provide several examples in which the minimal value of $q$ for which $[x_1, \dots, x_{q+1}]=0$ is an identity is explicitly determined. We also show that when we have a product of the form $E\otimes E_{i_1}\otimes \cdots \otimes E_{i_s}$, then the minimal value of $q$ is necessarily an even integer.
	
	In the second part of the paper, we take $K$ to be a field of characteristic $0$. By $\mathfrak{N}_p$ we denote the variety of all Lie nilpotent associative algebras of index at most $p$. Let $X_n = \{x_1, \dots, x_n\} \subset X$ and let $K\left\langle X_n \right\rangle$ denote the free associative algebra generated by $X_n$. By $P_n$ we denote the subspace of multilinear polynomials of degree $n$ in $K\left\langle X_n \right\rangle$ and by $\Gamma_n$ we denote the subspace of proper multilinear polynomials of degree $n$. The spaces $P_n$ and $\Gamma_n$ are naturally modules over the symmetric group $S_n$, where the action of the symmetric group is given by permutation of the variables. One of the main questions in the quantitative study of PI-algebras is to determine the $S_n$-cocharacter and codimension sequence of a given PI-algebra. In this paper, we are interested in the relatively free algebra of rank $n$ in the variety $\mathfrak{N}_p$. This algebra is defined as the quotient $F_n(\mathfrak{N}_p) = K\left\langle X_n \right\rangle / (I_{p+1} \cap K \left\langle X_n \right\rangle)$, where $I_{p+1}$ is the two-sided associative ideal in $K\left\langle X \right\rangle$ generated by all commutators of length $p+1$. The subspaces of multilinear and proper multilinear polynomials in $F_n(\mathfrak{N}_p)$, denoted respectively by $P_n(\mathfrak{N}_p)$ and $\Gamma_n(\mathfrak{N}_p)$, inherit the $S_n$-module structure of $P_n$ and $\Gamma_n$. The $S_n$-module structure of $\Gamma_n(\mathfrak{N}_p)$ is known for $p=1$, $2$, $3$, and $4$ (see \cite{V} for $p=3$ and \cite{SV} for $p=4$). In addition, for general values of $p$, we have the following special case of a theorem of Berele.
	\begin{theorem} [\cite{B}] \label{thm_Berele}	
		Let $M(\lambda)$ denote the irreducible $S_n$-module corresponding to the partition $\lambda$. Let
		$\Gamma_n(\mathfrak{N}_p) =  \bigoplus_{\lambda} m_{\lambda} M(\lambda).
		$	
		Then there exist integers $s_1$ and $s_2$ such that the non-zero multiplicities $m_{\lambda}$ are supported in diagrams of the form
		
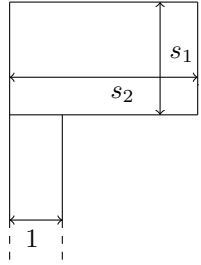
\begin{center}
			\begin{tikzpicture}
				\draw (0,0.5)--(2.5,0.5);
				\draw (2.5, -0.6) -- (2.5, -0.6);

				\draw (2.5, 0.5) -- (2.5, -1);
				\draw (0, -1) -- (2.5, -1);
				
				\draw[<->] (2, 0.5) -- (2, -1);
				\draw (2,-0.2) node[right] {{$s_1$}};
				
				\draw[<->] (0, -0.5) -- (2.5, -0.5);
				\draw (1.5,-0.5) node[below] {{$s_2$}};
				
				\draw (0, 0.5)--(0,-2.4);
				\draw (0.7,-1)--(0.7,-2.4);
			
			\draw [dashed] (0, -2.4)--(0,-3);
			\draw [dashed] (0.7,-2.4)--(0.7,-3);
			
				\draw[<->] (0, -2.4) -- (0.7, -2.4);
				\draw (0.3,-2.4) node[below] {{$1$}};
				
			\end{tikzpicture}
		\end{center}
Moreover, only the first column in the Young diagram of $\lambda$ can grow arbitrarily large when $n$ grows.
\end{theorem}			
	
Furthermore, in \cite{H}, we show that $s_2 = p-1$, i.e., there can be at most $p-1$ boxes in the first row of the Young diagram of $\lambda$ from Theorem \ref{thm_Berele}.

In this paper, we find many explicit classes of partitions  $\lambda \vdash n$, such that $M(\lambda)$ appears with non-zero multiplicity in the $S_n$-module decomposition of $\Gamma_n(\mathfrak{N}_p)$ for arbitrary values of $p$. More precisely, let $p = 2k$ or $p=2k+1$. Then, for the following partitions $\lambda$, we have that $M(\lambda)$ appears in $\Gamma_n(\mathfrak{N}_p)$:
\begin{itemize}
	\item [(i)] $(p-1, 1)$, $(p-1, p-1)$;
	\item [(ii)] $(a+2, 2^b, 1^c)$ such that $a \geq 0$, $b+c >0$, and $a+b+1 \leq 2k-2$;
	\item [(iii)] $(l+3, l+1, 1^{n-2l-4})$, $(l+2, l+2, 1^{n-2l-4})$, and $(l+2, l+1, 1^{n-2l-3})$, where $1 \leq l \leq k-2$;
	\item [(iv)] If $n$ is even, then we also have the partition $(1, 1, \dots, 1)$.
\end{itemize}

In addition, if $p = 2k+1$ we have also the following partitions $\lambda \vdash n$ for $n \leq 4k$:
\begin{itemize}
\item [(v)] $(a+2, 2^b)$ for $a \geq 0$, $b> 0$, and $a+b+1 = 2k$ or $a+b+1 = 2k-1$;
\item [(vi)] $(a+2, 2^b, 1)$ for $a \geq 0$, $b \geq 0$, and $a+b+1 = 2k-1$.
\end{itemize}

These results are given in Section \ref{sec_Sn_structure}. More precisely, part (i) is proven in Section \ref{subsec_part1}, parts (ii), (iv), (v), and (vi) are given in Section \ref{subsec_E}, and part (iii) is proven in Section \ref{subsec_Gamma5}. In the same section, 
using these results, we give lower bounds for the dimensions of $\Gamma_n(\mathfrak{N}_p)$ and $P_n(\mathfrak{N}_p$). For $p = 2k$ or $p = 2k+1$ with $k \geq 3$ these lower bounds are given in Corollary \ref{coro_final_bound_Gamma} and Corollary \ref{coro_final_bound_codim}. For $k = 2$, see Corollary \ref{coro_dim_FromE}.
	
	\section{Tensor product theorems for Lie nilpotent associative algebras}
	
	Let $K$ be any field. Let $L_p$ denote the Lie ideal in $K \left \langle X \right\rangle$ generated by all commutators of length $p$. In other words, $L_p$ is defined inductively by the formula $L_p = [L_{p-1}, K\left\langle X \right\rangle]$, where $L_1 =K\left\langle X \right\rangle$. Let $I_p = K\left\langle X \right\rangle \cdot L_p$ denote the two-sided associative ideal in $K \left \langle X \right\rangle$ generated by all commutators of length $p$. Hence, $I_p$ is the T-ideal generated by the element $[x_1, \dots, x_{p}]$. 
	
	First we recall the following two theorems on products of commutators which will be very useful in the sequel. Theorem \ref{thm_prod} was first proven by Latyshev \cite{L} in 1965 and then was independently rediscovered by Gupta and Levin \cite{GL} in 1983. The statement holds over any field $K$. Theorem \ref{thm_prod_odd} was proven by many authors for different restrictions on the characteristic of $K$ (see \cite{DeK2} for a short account on the existing results). The most general result that we give below is due to Deryabina and Krasilnikov \cite{DeK2}.
	\begin{theorem}[\cite{L}, \cite{GL}] \label{thm_prod} 
		For any integers $m\geq 2$ and $n \geq 2$ it holds that
		\[
		I_m I_n \subset I_{m+n-2}.
		\]
	\end{theorem}
	
	\begin{theorem} [\cite{DeK2}]\label{thm_prod_odd} 
		Let $\operatorname{char} K \neq 3$. If one of $m$ and $n$ is odd, then
		\[
		I_m I_n \subset I_{m+n-1}.
		\]	
	\end{theorem}		

We also need the following lemma from \cite{DeK2}, which can be found also in \cite{GrP} for $\operatorname{char} K \neq 2,3$.	
	\begin{lemma} [\cite{DeK2}] \label{lemma_In} 
		Let $\operatorname{char} K \neq 3$ and let $n \geq 2$. Let $u \in I_n$ and $x,y \in K\left\langle X \right\rangle$. Then
		\[
		[u, x, y] \in I_{n+2}.
		\]
		In other words, for all $n\geq 2$ it holds that
		\[
		[I_n, x,y] \subset I_{n+2}.
		\]
	\end{lemma}

\begin{remark} \label{rem_In} We notice that when $\operatorname{char} K = 3$, if we go over the proof of Lemma \ref{lemma_In} given in \cite{DeK2} and instead of the property $I_nI_3\subset I_{n+2}$, we use the property $I_nI_3 \subset I_{n+1}$ (which follows from Theorem \ref{thm_prod} for any characteristic of $K$), then we obtain the property
\[
[I_n, x,y] \subset I_{n+1}.
\]
\end{remark}
For completeness of the exposition, we will give a short proof of Remark \ref{rem_In} following the proof of Lemma \ref{lemma_In} given in \cite{DeK2}.
\begin{proof}
	Let $c \in L_n$, and $u,v,w \in K\left\langle X \right\rangle$. To proof the remark it suffices to show that
	\[
	[cu, v, w] \in I_{n+1}.
	\]
We have,
\[
[cu, v, w] = c[u,v,w] + [c,v][u,w] + [c,w][u,v] + [c,v,w]u.
\]
By Theorem \ref{thm_prod}, $c[u,v,w] \in I_{n+1}$. It is clear that $[c,v,w]u \in I_{n+2}$. Finally, as it is shown in \cite{DeK2},
\[
[c,v][u,w] + [c,w][u,v] = [v,c][w,u] + [v,u][w,c] + [[c,w], [u,v]] \in I_{n+2}
\]
since
\[
[v,c][w,u] + [v,u][w,c] = -[c, vw, u] + v[c,w,u] + [c,v,u]w.
\]
\end{proof}

We are now ready to prove the following statement, which is a generalization of Corollary 2.2 from \cite{DeK}.
	
	\begin{theorem} \label{thm_tensorProduct} 
	Let $K$ be any field. Let $G$ and $H$ be Lie nilpotent associative algebras over $K$ of indices respectively $p_1$ and $p_2$ such that $H$ satisfies also the identity $[x_1, x_2]\cdots [x_{2k-1}, x_{2k}] = 0$ for some $k \geq 2$. Then, $G\otimes H$ is again Lie nilpotent. Moreover, if $q$ is the minimal integer such that $[x_1, \dots, x_{q+1}]=0$ is an identity for $G\otimes H$, then $$\max \{p_1, p_2\} \leq q \leq 4(p_1 + p_2 + k -2).$$ When $\operatorname{char} K \neq 3$, there is a better upper bound for $q$, namely, $$\max \{p_1, p_2\} \leq q \leq 4(\left\lceil\frac{p_1-1}{2} \right\rceil + \left\lceil \frac{p_2-1}{2}\right\rceil + k).$$
	\end{theorem}
	
	\begin{proof}
	We first notice that if $a,c \in G$ and $b,d \in H$, then
		\begin{equation} \label{eq_multRule}
		\begin{aligned}
		&[a\otimes b, c\otimes d] = ac\otimes bd - ca\otimes db \pm ca\otimes bd = \\
		&[a,c]\otimes bd + ca\otimes[b,d].
		\end{aligned}
	\end{equation}
		Using the same notations as in \cite{DeK}, we set 
		\[
		c_{\ell} = [g_1\otimes h_1, \dots, g_{\ell}\otimes h_{\ell}],
		\]
		
		where $\ell \geq 2$, $g_i \in G$, and $h_i \in H$. 
		
		Then, Formula (\ref{eq_multRule}) gives the following expression for $c_\ell$
		\[
		c_\ell = \sum_{\varepsilon \in \{0,1\}^{\ell-1}} A^{(\ell)}_{\varepsilon} \otimes B^{(\ell)}_{\varepsilon},
		\]
		where $\varepsilon = (\varepsilon_1, \dots,\varepsilon_{\ell-2}, \varepsilon_{\ell-1}) \in \{0,1\}^{\ell-1}$, $A^{(1)} = g_1$, $B^{(1)} = h_1$ and we have the following recursive rule:
		\begin{itemize}
			\item if $\varepsilon_{\ell-1} = 0$, then
			\[
			A^{(\ell)}_{\varepsilon} = [A^{(\ell-1)}_{(\varepsilon_1, \dots, \varepsilon_{\ell-2})}, g_\ell] \quad B^{(\ell)}_{\varepsilon} = B^{(\ell-1)}_{(\varepsilon_1, \dots, \varepsilon_{\ell-2})}h_\ell;
			\]
			\item if $\varepsilon_{\ell-1} = 1$, then
			\[
			A^{(\ell)}_{\varepsilon} = g_\ell A^{(\ell-1)}_{(\varepsilon_1, \dots, \varepsilon_{\ell-2})} \quad B^{(\ell)}_{\varepsilon} = [B^{(\ell-1)}_{(\varepsilon_1, \dots, \varepsilon_{\ell-2})}, h_\ell]. 
			\]
		\end{itemize}
		
		For example, for $c_2$ and $c_3$ we obtain the following expressions:
		\[
		c_2 = [g_1, g_2]\otimes h_1h_2 + g_2g_1\otimes [h_1, h_2].
		\]	
		\begin{align*}
		&c_3 = [g_1, g_2, g_3] \otimes h_1h_2h_3 + g_3[g_1, g_2]\otimes[h_1h_2, h_3] + \\
		&[g_2g_1, g_3] \otimes [h_1, h_2]h_3 + g_3g_2g_1 \otimes [h_1, h_2, h_3].	
		\end{align*}	
		
		Next, we take $\ell > \max\{p_1,p_2\}$ and take $g\otimes h$ to be one of the summands $A^{(\ell)}_{\varepsilon} \otimes B^{(\ell)}_{\varepsilon}$ in $c_\ell$.   Then, $g$ is a polynomial in $g_1, \dots, g_\ell$ and $h$ is a polynomial in $h_1, \dots, h_\ell$. In other words, $g = f_1(g_1, \dots, g_\ell)$ and $h = f_2(h_1, \dots, h_\ell)$, where $f_1(x_1, \dots, x_\ell), f_2 (x_1, \dots, x_\ell) \in K\left\langle X \right\rangle$.
		
		Since $\ell > \max\{p_1,p_2\}$, the summands $[g_1, \dots, g_\ell]\otimes h_1 \cdots h_\ell$ and $g_\ell \cdots g_1 \otimes [h_1, \dots, h_\ell]$ vanish in $G\otimes H$. Therefore, there is at least one commutator in $g$ and at least one commutator in $h$. In other words, there exist $p'_1$ and $p'_2$ such that $f_1 \in I_{p'_1}$ and $f_2 \in I_{p'_2}$. With a slight abuse of notation we will say that $g \in I_{p'_1}$ and $h \in I_{p'_2}$. In addition, let $h = xu_1\dots u_{s}$ where $u_1, \dots, u_{s}$ are commutators and $x \in K\left\langle X \right\rangle$. Our goal is to take the commutator
		\[
		[g\otimes h, g_{\ell+1}\otimes h_{\ell+1}, g_{\ell+2}\otimes h_{\ell+2}, g_{\ell+3}\otimes h_{\ell+3}, g_{\ell+4}\otimes h_{\ell+4}]
		\]
		and show that it contains only summands of the form
		\begin{itemize}
			\item [(i)] $A^{(\ell+4)}_{\varepsilon'} \otimes B^{(\ell+4)}_{\varepsilon'}$, where $A^{(\ell+4)}_{\varepsilon'} \in I_{p'_1+2}$ (or in $I_{p'_1+1}$ when $\operatorname{char} K = 3$);
			\item [(ii)] $A^{(\ell+4)}_{\varepsilon'} \otimes B^{(\ell+4)}_{\varepsilon'}$, where $B^{(\ell+4)}_{\varepsilon'} \in I_{p'_2+2}$ (or in $I_{p'_2+1}$ when $\operatorname{char} K = 3$);
			\item [(iii)] $A^{(\ell+4)}_{\varepsilon'} \otimes B^{(\ell+4)}_{\varepsilon'}$, where
		$B^{(\ell+4)}_{\varepsilon'}$ has at least $s+1$ commutators.
	\end{itemize}
Here, $\varepsilon' \in \{0,1\}^{\ell+3}$ and $\varepsilon' = (\varepsilon, \varepsilon_{\ell}, \varepsilon_{\ell+1}, \varepsilon_{\ell+2}, \varepsilon_{\ell+3})$ for some $\varepsilon_{\ell}, \varepsilon_{\ell+1}, \varepsilon_{\ell+2}, \varepsilon_{\ell+3} \in \{0,1\}$. Properties (i), (ii), and (iii) will ensure that for large enough $\ell$ all terms in $c_\ell$ will vanish.

Using again Formula (\ref{eq_multRule}) we compute:
\[
[g\otimes h, g_{\ell+1}\otimes h_{\ell+1}] = [g, g_{\ell+1}]\otimes hh_{\ell+1} + g_{\ell+1}g\otimes [h, h_{\ell+1}].
\]

Then,
\begin{equation}\label{eq_commutatorCalculations}
\begin{aligned}
&[g\otimes h, g_{\ell+1}\otimes h_{\ell+1}, g_{\ell+2}\otimes h_{\ell+2}] = \\
&[g, g_{\ell+1}, g_{\ell+2}]\otimes hh_{\ell+1}h_{\ell+2} + g_{\ell+2}[g, g_{\ell+1}]\otimes [hh_{\ell+1},h_{\ell+2}] + \\
&[g_{\ell+1}g, g_{\ell+2}]\otimes [h, h_{\ell+1}]h_{\ell+2} + g_{\ell+2}g_{\ell+1}g \otimes [h, h_{\ell+1}, h_{\ell+2}].	
\end{aligned}
\end{equation}

By Lemma \ref{lemma_In} (or by Remark \ref{rem_In}), we have that $[g, g_{\ell+1}, g_{\ell+2}] \in I_{p'_1 + 2}$ (or in $I_{p'_1 +1}$) and similarly, $[h, h_{\ell+1}, h_{\ell+2}] \in I_{p'_2 + 2}$ (or in $I_{p'_2 + 1}$). Hence, the first and fourth terms in Equation (\ref{eq_commutatorCalculations}) satisfy respectively property (i) and property (ii). It remains to consider the second and the third summands.

When we commute the third summand of Equation (\ref{eq_commutatorCalculations}) with $g_{\ell+3}\otimes h_{\ell+3}$ we obtain
\begin{align*}
	&[[g_{\ell+1}g, g_{\ell+2}]\otimes [h,h_{\ell+1}]h_{\ell+2}, g_{\ell+3}\otimes h_{\ell+3}] = \\
	&[g_{\ell+1}g, g_{\ell+2}, g_{\ell+3}] \otimes [h, h_{\ell+1}]h_{\ell+2}h_{\ell+3} + \\
	&g_{\ell+3}[g_{\ell+1}g, g_{\ell+2}] \otimes [[h, h_{\ell+1}]h_{\ell+2}, h_{\ell+3}].
\end{align*}

Again by Lemma \ref{lemma_In} (or by Remark \ref{rem_In}), the first summand in the above equation satisfies property (i).
For the second summand, we notice the following.
\begin{align} \label{eq_third}
[[h, h_{\ell+1}]h_{\ell+2}, h_{\ell+3}] = [h, h_{\ell+1}][h_{\ell+2}, h_{\ell+3}] + [h, h_{\ell+1}, h_{\ell+3}]h_{\ell+2}.	
\end{align}
The second summand of Equation (\ref{eq_third}) is in $I_{p'_2 + 2}$ (or in $I_{p'_2 +1}$), so it remains to consider the first summand. We use here the fact that $h = xu_1\cdots u_{s}$.
\begin{equation} \label{eq_fourth}
\begin{aligned}
	&[xu_1\cdots u_{s}, h_{\ell+1}][h_{\ell+2}, h_{\ell+3}] = \\
	&[x, h_{\ell+1}]u_1\cdots u_{s} [h_{\ell+2}, h_{\ell+3}] + x[u_1\cdots u_{s}, h_{\ell+1}][h_{\ell+2}, h_{\ell+3}].
\end{aligned}
\end{equation}
The first summand of Equation (\ref{eq_fourth}) contains $s+2$ commutators and thus property (iii) is satisfied. For the second summand, we use that
\[
[u_1\cdots u_{s}, h_{\ell+1}] = \sum_{i = 1}^{s} u_1\cdots u_{i-1} [u_i, h_{\ell+1}]u_{i+1}\cdots u_{s}.
\]
Hence, the second summand of Equation (\ref{eq_fourth}) contains $s+1$ commutators and again property (iii) is satisfied.

It remains to consider the second summand of Equation  (\ref{eq_commutatorCalculations}).
When we commute it 
with $g_{\ell+3}\otimes h_{\ell+3}$ we obtain
\begin{align*}
	&[g_{\ell+2}[g, g_{\ell+1}]\otimes [hh_{\ell+1},h_{\ell+2}], g_{\ell+3}\otimes h_{\ell+3}] = \\
	&[g_{\ell+2}[g, g_{\ell+1}], g_{\ell+3}] \otimes [hh_{\ell+1}, h_{\ell+2}]h_{\ell+3} + \\
	&g_{\ell+3}g_{\ell+2}[g, g_{\ell+1}] \otimes [hh_{\ell+1}, h_{\ell+2}, h_{\ell+3}].
\end{align*}

Again by Lemma \ref{lemma_In} (or by Remark \ref{rem_In}), the second summand in the above equation satisfies property (ii).

It remains to consider the term $[g_{\ell+2}[g, g_{\ell+1}], g_{\ell+3}] \otimes [hh_{\ell+1}, h_{\ell+2}]h_{\ell+3}$ and we commute it with $g_{\ell+4}\otimes h_{\ell+4}$. We obtain,
\begin{align*}
	&[[g_{\ell+2}[g, g_{\ell+1}], g_{\ell+3}] \otimes [hh_{\ell+1}, h_{\ell+2}]h_{\ell+3}, g_{\ell+4}\otimes h_{\ell+4}] = \\
	&[g_{\ell+2}[g, g_{\ell+1}], g_{\ell+3}, g_{\ell+4}] \otimes [hh_{\ell+1}, h_{\ell+2}]h_{\ell+3}h_{\ell+4} + \\
	&g_{\ell+3}[g_{\ell+2}[g, g_{\ell+1}], g_{\ell+3}] \otimes [[hh_{\ell+1}, h_{\ell+2}]h_{\ell+3}, h_{\ell+4}].  
\end{align*}
For the first summand in the above formula we use again Lemma \ref{lemma_In} or Remark \ref{rem_In}. For the second summand we use the same reasoning as in Equations (\ref{eq_third}) and (\ref{eq_fourth}) to show that each term in the expression $[[hh_{\ell+1}, h_{\ell+2}]h_{\ell+3}, h_{\ell+4}]$ contains at least $s+1$ commutators or belongs to $I_{p'_2 + 2}$ ($I_{p'_2 + 1}$).	

The next step is to give a lower and an upper bound on the index $q$ of Lie nilpotency of $G\otimes H$. 
Since for $\ell \leq \max \{p_1, p_2\}$ at least one of the terms $[g_1, \dots, g_\ell]\otimes h_1 \cdots h_\ell$ and $g_\ell \cdots g_1 \otimes [h_1, \dots, h_\ell]$ in $c_{\ell}$ will not vanish, the lower bound for $q$ is clear.

It remains to give an upper bound for $q$. Notice that if we denote $I_1 = K\left\langle X \right\rangle$, then Lemma \ref{lemma_In} and Remark \ref{rem_In} will be valid for $I_1$. Properties (i), (ii), (iii) ensure that if we start with $c_1 = g_1 \otimes h_1$ and compute $c_\ell$ for $\ell = 4(p_1-1 + p_2-1 + k-1) + 1$, then the summands in $c_\ell$ will either vanish or will be of the form $A\otimes B$, where $A \in I_{p_1}$, $B \in I_{p_2}$ and $B$ consists of $k-1$ commutators. Therefore, if we make $4$ more iterations, i.e., compute $c_{\ell +4}$, then all terms will vanish. When $\operatorname{char} K \neq 3$, it is enough to compute $c_{\ell +4}$ for  $\ell = 4(\lceil\frac{p_1-1}{2} \rceil + \lceil \frac{p_2-1}{2}\rceil + k-1) + 1$. Hence, for the index $q$ of Lie nilpotency of $G\otimes H$ we obtain the bounds 
\begin{itemize}
\item When $\operatorname{char} K \neq 3$, then
\[
\max \{p_1, p_2\} \leq q \leq 4(\left\lceil\frac{p_1-1}{2} \right\rceil + \left\lceil \frac{p_2-1}{2}\right\rceil + k).
\]
\item When $\operatorname{char} K = 3$, then
\[
\max \{p_1, p_2\} \leq q \leq 4(p_1-1 + p_2-1 + k).
\]
\end{itemize}
	\end{proof}	

It follows from \cite{J}, that if $H$ is a finitely generated Lie nilpotent associative algebra, then the ideal $I_2$ is nilpotent, or in other words, $H$ satisfies an identity of the form $[x_1, x_2] \cdots [x_{2k-1}, x_{2k}]$ for some $k$. Another way to derive this fact is to notice that the variety $\mathfrak{N}_p$ is a non-matrix variety and hence every finitely generated algebra in $\mathfrak{N}_p$ satisfies a product of double commutators (see e.g. \cite{MPR}). Hence, an immediate corollary of Theorem \ref{thm_tensorProduct} is the following statement.

\begin{corollary}Let $K$ be any field. Let $G$ be any Lie nilpotent associative algebra and let $H$ be finitely generated and Lie nilpotent. Then $G\otimes H$ is again Lie nilpotent.
\end{corollary}	

As another corollary of Theorem \ref{thm_tensorProduct} we obtain the following statement which was first proven by Drensky in \cite{D}.
	
	\begin{corollary} [\cite{D}] \label{coro_tensorProduct}
		Let $\operatorname{char} K \neq 2$ and let $F =  E\otimes E_{i_1}\otimes \cdots \otimes E_{i_s}$ or $F = E_{j_1} \otimes E_{j_2} \otimes \cdots \otimes E_{j_t}$ for some integers $s \geq 0$ and $t \geq 1$ and integers $i_1, \dots, i_s, j_1, \dots, j_t \geq 2$. Then, there exists an integer $q \geq 2$ such that $F$ satisfies the identity $[x_1, \dots, x_{q+1}] = 0$.
	\end{corollary}
	
	\begin{proof}
		The proof follows by induction on $s$ and $t$. It is well-known that the Grassmann algebras $E$ and $E_{j_1}$ satisfy the identity $[x_1, x_2, x_3] = 0$. Next, assume that $F = F' \otimes E_r$, where $F' = E\otimes E_{i_1}\otimes \cdots \otimes E_{i_{s-1}}$ or $F' =  E_{j_1} \otimes E_{j_2} \otimes \cdots \otimes E_{j_{t-1}}$ and $r = i_s$ or $r = j_t$. Then, by the induction hypothesis there exists $q'$ such that $F'$ satisfies $[x_1, \dots, x_{q'+1}]=0$. Furthermore, it is well-known that if $E_r = E_{2k}$ or $E_r = E_{2k+1}$, then $E_r$ satisfies the identities $[x_1, x_2, x_3]$ and $[x_1, x_2] \cdots [x_{2k+1}, x_{2k+2}] = 0$. Therefore, by Theorem \ref{thm_tensorProduct}, there exists an integer $q$ such that $F'\otimes E_r$ satisfies $[x_1, \dots, x_{q+1}] = 0$.	
	\end{proof}	
	
The upper bound from Theorem \ref{thm_tensorProduct} is in general much bigger than the minimal value of $q$ for which $[x_1, \dots, x_{q+1}] = 0$ is an identity for $G\otimes H$. If we make further restrictions on the algebra $H$ from Theorem \ref{thm_tensorProduct} and we set $\operatorname{char} K \neq 3$, we are able to determine explicitly a better value of $q$. In many examples, this will turn out to be the minimal value of $q$ for which $[x_1, \dots, x_{q+1}]=0$ is an identity for $G\otimes H$.
	
	\begin{proposition} \label{prop_addE2}
		Let $\operatorname{char} K \neq 3$. Let $F = G\otimes H$ such that $G$ satisfies $[x_1, \dots, x_p] = 0$ and $H$ satisfies $[x_1, x_2, x_3] = 0$ and $[x_1, x_2][x_3, x_4] = 0$. Then,
		\begin{itemize}
			\item[(i)] If $p$ is even, then $G\otimes H$ satisfies $[x_1, \dots, x_{p+1}] = 0$.
			\item[(ii)] If $p$ is odd, then $G\otimes H$ satisfies $[x_1, \dots, x_{p+2}] = 0$.
		\end{itemize}	
	\end{proposition}
	
	\begin{proof}
		Let us denote again for $\ell \geq 2$, $g_i \in G$ and $h_i \in H$
		\[
		c_{\ell} = [g_1\otimes h_1, \dots, g_{\ell}\otimes h_{\ell}].
		\]
		
		We introduce also the following further notation: $A_2 = [g_1, g_2]$ and $A_i = [A_{i-1}, g_i]$ for $i > 2$. Then, having in mind that $H$ satisfies $[x_1, x_2][x_3, x_4]$, we prove by induction that for each $k \geq 3$
		\begin{equation} \label{eq_E2}
			\begin{aligned}
				&c_k = A_k \otimes h_1 \cdots h_k + g_kA_{k-1} \otimes [h_1\cdots h_{k-1}, h_k] + \\
				&+ \sum_{i=4}^k [g_{i-1}A_{i-2}, g_i, \dots, g_k] \otimes [h_1 \cdots h_{i-2}, h_{i-1}]h_i\cdots h_k +\\
				&+ [g_2g_1, g_3, \dots, g_k]\otimes [h_1, h_2]h_3\cdots h_k.
			\end{aligned}
		\end{equation}	
		
		We assume first that $p$ is even. Then, Equation (\ref{eq_E2}) yields
		\begin{align*}
			&c_{p+1} = A_{p+1} \otimes h_1 \cdots h_{p+1} + g_{p+1}A_{p} \otimes [h_1\cdots h_{p}, h_{p+1}] + \\
			&+ \sum_{\substack{i=4 \\ i=2j}}^{p} [g_{2j-1}A_{2j-2}, g_{2j}, \dots, g_{p+1}] \otimes [h_1 \cdots h_{2j-2}, h_{2j-1}]h_{2j}\cdots h_{p+1} +\\
			&+ \sum_{\substack{i=5 \\ i=2j+1}}^{p+1} [g_{2j}A_{2j-1}, g_{2j+1}, \dots, g_{p+1}] \otimes [h_1 \cdots h_{2j-1}, h_{2j}]h_{2j+1}\cdots h_{p+1} +\\
			&+ [g_2g_1, g_3, \dots, g_{p+1}]\otimes [h_1, h_2]h_3\cdots h_{p+1}.
		\end{align*}
		
		By definition, $A_p \in L_p$ and $A_{p+1} \in L_{p+1}$, hence the first two summands are equal to zero. Similarly, $[g_2g_1, g_3, \dots, g_{p+1}] \in L_p$, hence the last summand is also equal to zero.
		
		
		Next, we consider the following term
		\[
		[g_{2j}A_{2j-1}, g_{2j+1}] = [g_{2j}, g_{2j+1}]A_{2j-1} + g_{2j}[A_{2j-1}, g_{2j+1}].
		\]
		
		Since, $A_{2j-1} \in L_{2j-1}$ and $I_2I_{2j-1} \subset I_{2j}$, both summands in the above equation belong to $I_{2j}$. Then, by Lemma \ref{lemma_In}
		\[
		[g_{2j}A_{2j-1}, g_{2j+1}, g_{2j+2}, g_{2j+3}] \subset [I_{2j}, g_{2j+2}, g_{2j+3}] \subset I_{2j+2}.
		\]
		
		Since $p+1$ is odd, by repeatedly applying Lemma \ref{lemma_In}, we obtain
		\[
		[g_{2j}A_{2j-1}, g_{2j+1}, \dots, g_{p+1}] \in I_p.
		\]
		
		Hence, the fourth term in the expression for $c_{p+1}$ vanishes as well.
		
		Finally, by Lemma \ref{lemma_In}, 
		\[
		[g_{2j-1}A_{2j-2}, g_{2j},g_{2j+1}] \in I_{2j}.
		\]
		
		Since $p+1$ is odd, by repeatedly applying Lemma \ref{lemma_In}, we obtain
		\[
		[g_{2j-1}A_{2j-2}, g_{2j}, \dots, g_{p+1}] \in I_p.
		\]
		
		Hence, the third term in the expression for $c_{p+1}$ vanishes as well, and this shows that $c_{p+1} = 0$. 
		
		Next, assume that $p$ is odd. Then, Equation (\ref{eq_E2}) implies
		\begin{align*}
			&c_{p+2} = A_{p+2} \otimes h_1 \cdots h_{p+2} + g_{p+2}A_{p+1} \otimes [h_1\cdots h_{p+1}, h_{p+2}] + \\
			&+ \sum_{\substack{i=4 \\i=2j}}^{p+1} [g_{2j-1}A_{2j-2}, g_{2j}, \dots, g_{p+2}] \otimes [h_1 \cdots h_{2j-2}, h_{2j-1}]h_{2j}\cdots h_{p+2} +\\
			&+ \sum_{\substack{i = 5 \\ i=2j+1}}^{p+2} [g_{2j}A_{2j-1}, g_{2j+1}, \dots, g_{p+2}] \otimes [h_1 \cdots h_{2j-1}, h_{2j}]h_{2j+1}\cdots h_{p+2} +\\
			&+ [g_2g_1, g_3, \dots, g_{p+2}]\otimes [h_1, h_2]h_3\cdots h_{p+2}.
		\end{align*}
		
		Similarly to the even case, we show that all terms in the above sum vanish, and hence for odd $p$ we have $c_{p+2} = 0$.	
	\end{proof}

In the next statements, we give several examples in which we are able to determine explicitly the minimal value of $q$ from Corollary \ref{coro_tensorProduct}. For the statements that depend on Proposition \ref{prop_addE2} we have the additional condition that $\operatorname{char} K \neq 3$, while for the other statements we have only the usual requirement for Grassmann algebras that $\operatorname{char} K \neq 2$. It is interesting to know whether the condition $\operatorname{char} K \neq 3$ can be dropped from all statements.

\begin{proposition} \label{prop_equal}
	Let $\operatorname{char} K \neq 2$. Let $F = E_{2k} \otimes E_{2k}$, $F = E_{2k+1} \otimes E_{2k+1}$, $F = E_{2k} \otimes E_{2k+1}$, or $F = E_{2k+1} \otimes E_{2k}$. Then, $F$ satisfies the identity $[x_1, \dots, x_{2k+2}]$. Furthermore, $F$ does not satisfy the identity $[x_1, \dots, x_{2k+1}]=0$.
\end{proposition}

\begin{proof}
	The conditions of the proposition imply that $F = G\otimes H$ such that both $G$ and $H$ satisfy $[x_1, x_2, x_3] = 0$ and $[x_1, x_2] \cdots [x_{2k+1}, x_{2k+2}] = 0$. Let us again denote for $\ell \geq 2$
	\[
	c_{\ell} = [g_1\otimes h_1, \dots, g_{\ell}\otimes h_{\ell}],
	\]
	where $g_i \in G$ and $h_i \in H$. Then, by Lemma 2.1 from \cite{DeK}, we have
	\begin{align*}
		&c_{2k+2} &= [g_1, g_2][g_3, g_4] \cdots [g_{2k+1}, g_{2k+2}]\otimes [h_1h_2, h_3] [h_4, h_5] \cdots [h_{2k}, h_{2k+1}]h_{2k+2} + \\
		&&+ [g_2g_1, g_3][g_4, g_5] \cdots [g_{2k}, g_{2k+1}]g_{2k+2}\otimes [h_1, h_2] [h_3, h_4] \cdots [h_{2k+1}, h_{2k+2}]. 
	\end{align*}
	
	In the first summand we have $[g_1, g_2]\cdots [ g_{2k+1}, g_{2k+2}] = 0$ and in the second summand we have $[h_1, h_2] \cdots [h_{2k+1}, h_{2k+2}] = 0$. Hence, $c_{2k+2} = 0$.
	
	The next step is to show that $E_{2k} \otimes E_{2k}$ does not satisfy $c_{2k+1} = 0$. Let $\{1, e_1, \dots, e_{2k}\}$ denote the generators of the first copy of $E_{2k}$ and let $\{1, f_1, \dots, f_{2k}\}$ denote the generators of the second copy of $E_{2k}$. Then,
	\[
	[e_1\otimes 1, e_2 \otimes f_1, e_3\otimes f_2, \dots, e_{2k}\otimes f_{2k-1}, 1 \otimes f_{2k}] = 2^{2k}e_1\cdots e_{2k}\otimes f_1\cdots f_{2k} \neq 0.
	\]	
\end{proof}		

We recall that in \cite{DeK} it was proven that $E\otimes E_r$, where $r = 2k$ or $r= 2k+1$, satisfies the identity $[x_1, \dots, x_{2k+3}]$.	
	Therefore, Proposition \ref{prop_addE2} together with \cite{DeK} and Proposition \ref{prop_equal} imply the following corollary.
	
	\begin{corollary}\label{coro_exact_q} Let $\operatorname{char} K \neq 2,3$.
		Let $r = 2k$ or $r = 2k+1$. Then,
		\begin{itemize}
			\item $E\otimes E_{r} \otimes E_2$ and $E \otimes E_r \otimes E_3$ satisfy the identity $[x_1, \dots, x_{2k+5}]=0$;
			\item $E_r \otimes E_r \otimes E_2$ and $E_r \otimes E_r \otimes E_3$ satisfy $[x_1, \dots, x_{2k+3}]=0$.
		\end{itemize}
	\end{corollary}

\begin{corollary} \label{coro_tensorByE2} 
	Let $\operatorname{char} K \neq 2,3$. Then,
		$E\otimes \underbrace{E_2\otimes \cdots \otimes E_2}_{k}$ and $E\otimes \underbrace{E_3\otimes \cdots \otimes E_3}_{k}$ satisfy $[x_1, \dots, x_{2k+3}] = 0$. Moreover, for $\operatorname{char} K \neq 2$ it holds that $E\otimes \underbrace{E_2\otimes \cdots \otimes E_2}_{k}$ and $E\otimes \underbrace{E_3\otimes \cdots \otimes E_3}_{k}$ do not satisfy $[x_1, \dots, x_{2k+2}] = 0$.
	\end{corollary}	

\begin{proof}
	It is enough to show that $E\otimes \underbrace{E_2\otimes \cdots \otimes E_2}_{k}$ does not satisfy $[x_1, \dots, x_{2k+2}] = 0$. Let $\{1, e_1, e_2, \dots\}$ denote the generators of $E$ and let $\{1, f_1, f_2\}$ denote the generators of $E_2$. We take the following elements from $E\otimes \underbrace{E_2\otimes \cdots \otimes E_2}_{k}$:
	
	\begin{align*}
		a_1 = e_1 \otimes 1 \otimes \cdots \otimes 1; \quad a_{2k+2} = e_{2k+2} \otimes 1 \otimes \cdots \otimes 1;
\end{align*}
For $1 \leq i \leq k$,
\begin{align*}	
		&a_{2i} = e_{2i} \otimes 1 \otimes \cdots \otimes 1 \otimes f_1 \otimes 1 \otimes \cdots \otimes 1 \\
		&a_{2i+1} = e_{2i+1} \otimes 1 \otimes \cdots \otimes 1 \otimes f_2 \otimes 1 \otimes \cdots \otimes 1,
\end{align*}
where $f_1$ and $f_2$ appear at position $i+1$ in the tensor product. 

Then,
\[
[a_1, \dots, a_{2k+2}] = 2^{2k+1} e_1 \cdots e_{2k+2} \otimes f_1f_2 \otimes \cdots \otimes f_1f_2 \neq 0.
\]
\end{proof}	

Therefore, for $\operatorname{char} K \neq 2,3$ the minimal value of $q$ such that $[x_1, \dots, x_q] = 0$ is an identity for $E\otimes \underbrace{E_2\otimes \cdots \otimes E_2}_{k}$ is $q = 2k+3$. An interesting question here is what is the minimal value of $q$ when $\operatorname{char} K = 3$, and in particular whether it is again $q = 2k+3$, or it is bigger. 
	
More generally, we have the following statement.
\begin{corollary}
	Let $\operatorname{char} K \neq 2,3$. Then,
	$E\otimes E_{2k} \otimes \underbrace{E_2\otimes \cdots \otimes E_2}_{l}$  satisfies $[x_1, \dots, x_{2k+2l + 3}] = 0$. Moreover, for $\operatorname{char} K \neq 2$ it holds that $E\otimes E_{2k} \otimes \underbrace{E_2\otimes \cdots \otimes E_2}_{l}$  does not satisfy $[x_1, \dots, x_{2k+ 2l + 2}] = 0$.
\end{corollary}

\begin{proof}
	Let $\{1, e_1, e_2, \dots\}$, $\{1, f_1, \dots, f_{2k}\}$, and $\{1, g_1, g_2\}$ denote the generators respectively of $E$, $E_{2k}$, and $E_2$.
	To show that $E\otimes E_{2k} \otimes \underbrace{E_2\otimes \cdots \otimes E_2}_{l}$  does not satisfy $[x_1, \dots, x_{2k+ 2l + 2}] = 0$ we use the following elements:
	
		\begin{align*}
	&	a_1 = e_1 \otimes 1 \otimes 1 \otimes \cdots \otimes 1; \\
	&a_2 = e_2 \otimes f_1 \otimes 1 \otimes \cdots \otimes 1; \\
	&\dots \dots\\
	&a_{2k+1} = e_{2k+1} \otimes f_{2k} \otimes 1 \otimes \cdots \otimes 1.
	\end{align*}
	For $1 \leq i \leq l$,
	\begin{align*}	
		&a_{2k + 2i} = e_{2k+ 2i} \otimes 1 \otimes  1 \otimes \cdots \otimes 1 \otimes g_1 \otimes 1 \otimes \cdots \otimes 1; \\
		&a_{2k+2i+1} = e_{2k+2i+1} \otimes 1 \otimes 1 \otimes \cdots \otimes 1 \otimes g_2 \otimes 1 \otimes \cdots \otimes 1,
	\end{align*}
	where $g_1$ and $g_2$ appear at position $i+2$ in the tensor product. Finally,
	\[
	a_{2k+2l+2} = e_{2k+2l+2} \otimes 1 \otimes 1 \otimes \cdots \otimes 1.
	\]
	
	\end{proof}
	
	For the algebra $\underbrace{E_2 \otimes \cdots \otimes E_2}_{l}$ we have the following statements.
	
	\begin{lemma} \label{lemma_productE2}
		Let us denote the generators of $E_2$ by $e_1$ and $e_2$ and consider the algebra $F_l = \underbrace{E_2 \otimes \cdots \otimes E_2}_{l}$. Let $2 \leq j \leq l+1$. Then, for any elements $a_1, \dots, a_j \in F_l$, we have that if $[a_1, \dots, a_j] \neq 0$, then $[a_1, \dots, a_j]$ is a linear combination of pure tensors from $F_l$ such that in each pure tensor at least $j-1$ positions are equal to $e_1e_2$.
	\end{lemma}	
	
	\begin{proof}
		We use induction on $j$. First we consider the commutator $[a_1, a_2]$, where $a_1$ and $a_2$ are arbitrary pure tensors from  $F_l$. Let us denote the element from $E_2$ which appears at position $i$ in $a_1$ by $a_1[i]$, and similarly for $a_2$. If $[a_1, a_2] \neq 0$, then there exists an integer $i$ between $1$ and $l$ such that $a_1[i] \neq a_2[i]$, $a_1[i] \neq 1$, and $a_2[i]\neq 1$. Then, $(a_1a_2)[i] = k e_1e_2$ for some nonzero constant $k$. This proves the statement for $j=2$. 
		
		Assume that the statement holds for $j-1$ and take arbitrary pure tensors $a_1, \dots, a_{j-1}, a_j \in F_l$ such that $[a_1, \dots, a_j] \neq 0$. Then, for each pure tensor $u$ from $[a_1, \dots, a_{j-1}]$ there exist positions $i_1, \dots, i_{j-2}$ such that $u[i_1] = \dots = u[i_{j-2}] = e_1e_2$. Now we consider the commutator $[u, a_j]$. If $[u, a_j] \neq 0$, there exists $i \neq i_1, \dots, i_{j-2}$ such that $u[i] \neq 1$, $a_j[i] \neq 1$, and $u[i] \neq a_j[i]$. Hence, as before, $(ua_j)[i] = k e_1 e_2$ for some nonzero constant $k$. This proves the statement.
	\end{proof}	
	
	\begin{proposition} \label{prop_productE2} The algebra  $F_l = \underbrace{E_2 \otimes \cdots \otimes E_2}_{l}$ satisfies the identity $[x_1, \dots, x_{l+2}]$.
	\end{proposition}
	\begin{proof}
		
		By Lemma \ref{lemma_productE2}, for arbitrary pure tensors $a_1, \dots, a_{l+1} \in F_l$, it holds that
		\[
		[a_1, \dots, a_{l+1}] = k(e_1e_2\otimes \dots \otimes e_1e_2),
		\]	
		
		for some constant $k$. Then, for any element $a \in F_l$, $[a_1, \dots, a_{l+1}, a] = 0$.	
	\end{proof}	
	
	In the end of this section, we show that if $F = E \otimes E_{i_1} \otimes \cdots \otimes E_{i_s}$, then the minimal value of $q$ for which $[x_1, \dots, x_{q+1}]$ is an identity for $F$ is necessarily even.
	
	\begin{proposition} \label{prop_oddq}
		Let $\operatorname{char} K \neq 2,3$ and let $F = E \otimes E_{i_1} \otimes \cdots \otimes E_{i_s}$ for $s \geq 0$ and any choice of integers $i_1, \dots, i_s$ greater or equal to $2$. If $F$ does not satisfy the identity $[x_1, \dots, x_{2k+1}]$ for some positive integer $k$, then $F$ does not satisfy the identity $[x_1, \dots, x_{2k+2}]$ either.
	\end{proposition}
	
	\begin{proof}
		Let $a_1, \dots, a_{2k+1}$ be elements from $F$ such that $[a_1, \dots, a_{2k+1}] \neq 0$. We write the expression $[a_1, \dots, a_{2k+1}]$ as a finite sum of pure tensors of the form
		\[
		[a_1, \dots, a_{2k+1}]	= \sum_j u_{j0} \otimes u_{j1} \otimes \cdots \otimes u_{js},	
		\]
		where $u_{j0} \in E$ and for each $l\geq 1$, $u_{jl} \in E_{i_l}$. Moreover, the elements $u_{j1} \otimes \cdots \otimes u_{js}$ are different for different values of $j$. Let us denote as usual the generators of $E$ by $e_1, e_2, \dots, e_n, \dots$. For each $j$, the element $u_{j0}$ is a linear combination of finitely many generators. We choose elements $e_{r1} \neq e_{r2}$ such that they do not appear in $u_{j0}$ for any $j$ and set $b_1 = e_{r1} \otimes 1 \otimes \cdots \otimes 1 \in F$ and $b_2 = e_{r2} \otimes 1 \otimes \cdots \otimes 1 \in F$. Then we claim that
		\begin{align} \label{eq_product_ab}
			[a_1, \dots, a_{2k+1}][b_1, b_2] \neq 0.	
		\end{align}
		Indeed, we have that
		\begin{align*}
			&	[a_1, \dots, a_{2k+1}][b_1, b_2] = 
			\sum_j u_{j0}[e_{r1}, e_{r2}] \otimes u_{j1} \otimes \cdots \otimes u_{js} = \\
			&\sum_j 2u_{j0}e_{r1} e_{r2} \otimes u_{j1} \otimes \cdots \otimes u_{js}.	
		\end{align*}
		
		Due to the choice of $e_{r1}$ and $e_{r2}$, for every $j$, we have that $u_{j0}e_{r1} e_{r2} \neq 0$. Hence (\ref{eq_product_ab}) holds. 
		
		Furthermore, by Theorem \ref{thm_prod_odd}, $[a_1, \dots, a_{2k+1}][b_1, b_2] \in I_{2k+2}$ and thus it is a consequence of the long commutator $[x_1, \dots, x_{2k+2}]$. This completes the proof of the statement.
	\end{proof}

Below we give an example of a tensor product of Lie nilpotent associative algebras that are not Grassmann algebras. 

\begin{example} Let $k \geq 3$ and let $K$ be an infinite field with $\operatorname{char} K \neq 3$. Let $J = \sum_{i=1}^{k-1} e_{i, i+1}$ where $e_{i,j}$ denotes the $k\times k$ matrix with $i,j$-th entry equal to $1$ and all other entries equal to $0$. Let $E = E_{k \times k}$ denote the identity matrix. In \cite{GiMP}, the authors define the algebras $N_k$ in the following way:
	\[
	N_k = \operatorname{span} \{E, J, J^2, \dots, J^{k-2}, e_{12}, e_{13}, \dots, e_{1k} \}.
	\] 
	
In the same paper, among the other important properies of the algebras $N_k$, it is shown that a basis for the identities of $N_k$ is given by the polynomials 
\begin{align*}
	[x_1, \dots, x_k] \text{ and } [x_1, x_2][x_3, x_4].
\end{align*}

Hence, Proposition \ref{prop_addE2} implies that the tensor product $N_k \otimes \underbrace{N_3 \otimes \cdots \otimes N_3}_{l}$ is a Lie nilpotent associative algebra and satisfies the identity $[x_1, \dots, x_{k+2l}]$ if $k$ is odd and $[x_1, \dots, x_{k+2l-1}]$ if $k$ is even. 
\end{example}	
	
\begin{remark}
	In \cite{D}, Drensky conjectures that over a field of characteristic zero, for any value of $p$, the variety $\mathfrak{N}_p$ is generated by a finite set of algebras of the type $E \otimes E_{i_1} \otimes \cdots \otimes E_{i_s}$ and $ E_{j_1} \otimes \cdots \otimes E_{j_r}$. Furthermore, in \cite{HdM}, the author and de Mello make the conjecture that the varieties $\mathfrak{N}_{2k}$ and $\mathfrak{N}_{2k+1}$ are asymptotically equivalent. The two conjectures are still open, but it follows from Proposition \ref{prop_oddq} that a positive answer to Drensky's conjecture would imply a positive answer to the conjecture of the author and de Mello.	
\end{remark}

	\section{Some concrete modules in the $S_n$-module structure of $\Gamma_n(\mathfrak{N}_p)$ for arbitrary $n$ and $p$} \label{sec_Sn_structure}
	
	In this section, we take $K$ to be a field of characteristic zero. We start by introducing the necessary definitions coming from the theory of algebras with polynomial identities. Then we describe concrete irreducible $S_n$-modules that appear with nonzero multiplicity in the decomposition of the subspace of proper multilinear polynomials of degree $n$ in $F_n(\mathfrak{N}_p)$. As a consequence we obtain lower bounds on the dimensions of the space of proper multilinear polynomials and the space of multilinear polynomials of degree $n$ in $F_n(\mathfrak{N}_p)$.
	
	\subsection{Preliminaries}
	 We say that a commutator $[x_{i_1}, \dots, x_{i_r}] \in K\langle X \rangle$ is {\it pure} if all the entries in the commutator are elements of the set $X$. A polynomial in $K\langle X \rangle$ is called {\it proper} if it is a linear combination of products of pure commutators. One important property of proper polynomials is that all identities of an algebra with a unit follow from its proper ones (\cite{D2}, Proposition 4.3.3). We recall that by $P_n$ we denote the vector space of multilinear polynomials of degree $n$ in  $K\langle X_n \rangle$ and by $\Gamma_n$ the subspace of $P_n$ of proper polynomials. To any polynomial from $K \left \langle X \right\rangle$ one can associate its complete multilinearization,
	 which is a polynomial in $P_n$. It is well known that for an algebra over a field of characteristic zero, a polynomial is an identity if and only if its complete multilinearization is also an identity. Hence, in the study of polynomial identities for unital algebras over fields of characteristic zero, we can restrict ourselves to proper multilinear polynomials.

For an algebra $A$, we denote by $T(A)$ the set of all polynomial identities of $A$. $T(A)$ is an ideal in $K \left\langle X \right\rangle$ that has the additional property that it is closed under all $K$-algebra endomorphisms of $A$. Such ideals are called {\it T-ideals}. We also take the natural quotient map
\begin{align} \label{eq_quotient_map}
K \left\langle X_n \right\rangle \rightarrow K \left\langle X_n \right\rangle/(K \left\langle X_n \right\rangle \cap T(A))
\end{align}
and define $P_n(A)$ (respectively, $\Gamma_n(A)$) to be the image of $P_n$ (resp., $\Gamma_n$) under this map.
Then the dimensions
\[
c_n(A) = \dim P_n(A) \quad \text{ and } \quad \gamma_n(A) = \dim \Gamma_n(A)
\]
are called respectively the $n$-th \textit{codimension} and the $n$-th \textit{proper codimension} of $A$ (or of $T(A)$).

The following relations between $c_n(A)$ and $\gamma_n(A)$ are well-known.

\begin{proposition}[\cite{D}] \label{prop_codim_propcodim}
	Let $A$ be a PI-algebra.
	\begin{itemize}
		\item [(i)] The codimension sequence $c_n(A)$ and the proper codimension sequence $\gamma_n(A)$ are related by the equation
		\[
		c_n(A) = \sum_{l = 0}^n \binom{n}{l}\gamma_l(A).
		\]
		\item[(ii)] The codimension series $c(A,t) = \sum_n c_n(A)t^n$ and the proper codimension series $\gamma(A,t) = \sum_n \gamma_n(A) t^n$ satisfy the equation
		\[
		c(A,t) = \frac{1}{1-t}\gamma\left (A, \frac{t}{1-t}\right).
		\]
	\end{itemize}
\end{proposition}

When $A$ is such that $T(A) = I_{p+1}$, then Equation (\ref{eq_quotient_map}) has the form
\[
K \left\langle X_n \right\rangle \rightarrow F_n(\mathfrak{N}_p) = K \left\langle X_n \right\rangle/(K \left\langle X_n \right\rangle \cap I_{p+1}).
\]
In this case, we use the notations $P_n(A) = P_n(\mathfrak{N}_p)$ and $\Gamma_n(A) = \Gamma_n(\mathfrak{N}_p)$ and similarly for the codimensions $c_n(A) = c_n(\mathfrak{N}_p)$ and $\gamma_n(A) = \gamma_n(\mathfrak{N}_p)$. 
	
	\subsection{Submodules of $\Gamma_n(\mathfrak{N}_p)$ corresponding to the partitions $(p-1, 1)$ and $(p-1, p-1)$} \label{subsec_part1}
	
	It is shown in \cite{H}, that if $M(\lambda)$ is an irreducible $S_n$-submodule of $\Gamma_n(\mathfrak{N}_{p})$ corresponding to a partition $\lambda = (\lambda_1, \dots, \lambda_k)$ with $|\lambda| = n$, then $\lambda_1 \leq p-1$. This bound is sharp, since the linearization of $[x_2, \underbrace{x_1, \dots, x_1}_{p-1}]$ generates an irreducible $S_{p}$-submodule of $\Gamma_p(\mathfrak{N}_{p})$ that corresponds to the partition $(p-1, 1)$. In this section, we prove that for any $p$ there is also a submodule of $\Gamma_{2p-2}(\mathfrak{N}_{p})$ that corresponds to the partition $(p-1, p-1)$. More precisely, we have the following proposition.
	
	\begin{proposition} \label{prop_module}
		The linearization of the polynomial $[x_1, x_2]^{p-1}$ generates an irreducible $S_{2p-2}$-submodule of $\Gamma_{2p-2}(\mathfrak{N}_{p})$ that corresponds to the partition $(p-1,p-1)$.
	\end{proposition} 
	
	\begin{proof}
		We will show that the polynomial $u = [x_1, x_2]^{p-1}$ is not an identity for $\underbrace{E_2\otimes \cdots \otimes E_2}_{p-1}$. Hence, by Proposition \ref{prop_productE2} it will follow that $u$ is not an element of $I_{p+1}$ and hence it is not an identity for $F(\mathfrak{N}_{p})$.
		
		We set $x_s = \sum_{i = 1}^{p-1} 1 \otimes \cdots \otimes 1 \otimes e_s \otimes 1 \otimes \cdots 1$, where $e_s$ appears at position $i$ in the tensor product and $s = 1,2$. Then,
		\[
		u = \underbrace{e_1e_2 \otimes \dots \otimes e_1e_2}_{p-1} \neq 0.
		\]
		Hence, $u$ is not an identity for $\underbrace{E_2\otimes \cdots \otimes E_2}_{p-1}$. 
	\end{proof}
	
By the hook formula,
	\[
	\dim M(p-1, p-1) = \frac{1}{p-1}\binom{2p-2}{p}.
	\]
	
	\subsection{Submodules coming from $E\otimes E_{2k-2}$ and $E_{2k}\otimes E_{2k}$} \label{subsec_E}
	We recall that the algebra $E\otimes E_{2k-2}$ satisfies the identity $[x_1, \dots, x_{2k+1}]$ (see \cite{DeK}) and the algebra $E_{2k} \otimes E_{2k}$ satisfies the identity $[x_1, \dots, x_{2k+2}]$ (Proposition \ref{prop_equal}). In \cite{DiD}, the authors determine the $S_n$-module structure of the algebras $E\otimes E_{2l}$ and $E_{2m} \otimes E_{2l}$. In particular, they prove the following theorems.
	
	\begin{theorem} [\cite{DiD}] \label{thm_DiD}
		Let $l \geq 1$. Then,
		\[ 	
		\Gamma_n(E\otimes E_{2l}) = \sum M(a+2, 2^b, 1^c) + \varepsilon_n M(1^n),
		\]
		where the sum is over all partitions $(a+2, 2^b, 1^c)$ of $n$ such that $a \geq 0$, $b+c > 0$, and $a + b + 1 \leq 2l$; $\varepsilon_n = 1$ for even $n$ and $\varepsilon_n = 0$ for odd $n$.
	\end{theorem}
	
	\begin{theorem} [\cite{DiD}] \label{thm_DiD2}
		Let $h_{ij}(\lambda)$ denote the $(i,j)$-th hook of the Young diagram of the partition $\lambda$. If $m \geq l \geq 1$, then
		\[ 	
		\Gamma_n(E_{2m}\otimes E_{2l}) = \sum M(a+2, 2^b, 1^c) + \varepsilon_n M(1^n),
		\]
		where $\varepsilon_n = 1$ for even $n \leq 2(m+l)$ and $\varepsilon_n = 0$ otherwise and the sum is over all partitions $\lambda = (a+2, 2^b, 1^c)$ of $n$ such that $a \geq 0$, $b+c > 0$, and $h_{12}(\lambda) = a + b + 1 \leq 2l$ and one of the following conditions holds:
		\begin{itemize}
			\item [(i)] $h_{11}(\lambda) + h_{12}(\lambda) -1 = 2a + 2b + c + 2 < 2(m+l)$;
			\item  [(ii)] $h_{11}(\lambda) + h_{12}(\lambda) -1 = 2(m+l)$ and $h_{12}(\lambda) \equiv 0(\operatorname{mod} 2)$.
		\end{itemize}

	\end{theorem}
	
	Theorem \ref{thm_DiD} applied to $E\otimes E_{2k-2}$ implies that
	\[
	\Gamma_n(\mathfrak{N}_{2k+1}) \supseteq \Gamma_n(\mathfrak{N}_{2k}) \supset \sum M(a+2, 2^b, 1^c) + \varepsilon_n M(1^n)
	\]
	with $a \geq 0$, $b+c > 0$, and $a + b + 1 \leq 2k-2$; $\varepsilon_n = 1$ for even $n$ and $\varepsilon_n = 0$ for odd $n$.
	Furthermore, if we apply Theorem \ref{thm_DiD2} to $E_{2k} \otimes E_{2k}$ we obtain that for $n \leq 4k$, $\Gamma_n(\mathfrak{N}_{2k+1})$ is strictly larger than $\Gamma_n(\mathfrak{N}_{2k})$ and contains also the modules $M(\lambda)$ corresponding to the partitions
	\begin{itemize}
		\item $\lambda = (a+2, 2^b)$ for $a \geq 0$, $b> 0$, and $a+b+1 = 2k$ or $a+b+1 = 2k-1$;
		\item $\lambda = (a+2, 2^b, 1)$ for $a \geq 0$, $b \geq 0$, and $a+b+1 = 2k-1$.
	\end{itemize}	
	
	In the last part of this subsection, we will compute lower bounds for the codimension sequences of the varieties $\mathfrak{N}_{2k}$ and $\mathfrak{N}_{2k+1}$.
	
	In \cite{DiD}, the authors compute the proper codimension sequences and the codimension sequence of $E \otimes E_{2l}$ and $E_{2m}\otimes E_{2k}$. Namely, they have the following statements.
	
	\begin{proposition} [\cite{DiD}] \label{prop_codimE}
		Let $m \geq l \geq 1$. For the proper codimension sequences of $E \otimes E_{2l}$ and $E_{2m}\otimes E_{2k}$ we have the following
		\begin{itemize} 
			\item $\gamma_{2(m+l)}(E_{2m}\otimes E_{2k}) > 0$ and $\gamma_{n}(E_{2m}\otimes E_{2k}) = 0$ for $n > 2(m+l)$.
			\item For $n \geq 2l+2$, 
			$\gamma_{n}(E\otimes E_{2l})$ is a polynomial with rational coefficients of degree $2l$ in $n$ with leading term
			\[
			\frac{1}{(2l)!}\sum_{p = 0}^{2l}\dim M(2l-p, 1^p)  = \frac{2^{2l-1}}{(2l)!}.
			\] 
		\end{itemize}
	\end{proposition}	
	\begin{theorem} [\cite{DiD}] \label{thm_codimE}
		Let $m \geq l \geq 1$. 	
		\begin{itemize}
			\item [(i)] The codimension sequence $c_n(E_{2m}\otimes E_{2l})$ is a polynomial with rational coefficients of degree $2(m+l)$ in $n$.
			\item[(ii)]	Let $n > 0$. Then
			\[
			c_n(E \otimes E_{2l}) = 2^{n-1} \xi_l(n) + \eta_l(n),
			\]	
			where $\xi_l(n)$ and $\eta_l(n)$ are polynomials with rational coefficients in $n$ such that $\deg \xi_l(n) = 2l$, $\deg \eta_l(n) \leq 2l+1$ 
			and the leading term of $\xi_l(n)$ is equal to $((2l)!)^{-1}$.		
		\end{itemize}
	\end{theorem}

	We obtain the following immediate corollary from Proposition \ref{prop_codimE} and Theorem \ref{thm_codimE}.
	
	\begin{corollary} \label{coro_dim_FromE} Let $k \geq 1$. Then,
		\begin{itemize}
			\item [(i)] For $n \geq 2k$, \[
			\gamma_n(\mathfrak{N}_{2k+1}) \geq \gamma_n(\mathfrak{N}_{2k}) \geq p(n),
			\]
			where $p(n)$ is a polynomial with rational coefficients of degree $2k-2$ in $n$ with leading coefficient $\frac{2^{2k-3}}{(2k-2)!}$.	 
			\item[(ii)]
			\[
			c_n(\mathfrak{N}_{2k+1}) \geq c_n(\mathfrak{N}_{2k}) \geq  2^{n-1} \xi(n) + \eta(n),
			\]	
			where $\xi(n)$ and $\eta(n)$ are polynomials with rational coefficients in $n$ such that $\deg \xi(n) = 2k-2$, $\deg \eta(n) \leq 2k-1$ 
			and the leading term of $\xi(n)$ is equal to $((2k-2)!)^{-1}$.
		\end{itemize}	
		Furthermore, for $n \leq 4k$ we have that $c_n(\mathfrak{N}_{2k+1})$ (respectively, $\gamma_n(\mathfrak{N}_{2k+1})$) is strictly bigger that $c_n(\mathfrak{N}_{2k})$ (resp., $\gamma_n(\mathfrak{N}_{2k})$).
	\end{corollary}

	
\subsection{Submodules of $\Gamma_n(\mathfrak{N}_{2k})$ and $\Gamma_n(\mathfrak{N}_{2k+1})$ coming from $\Gamma_n(\mathfrak{N}_4)$} \label{subsec_Gamma5}
	In the paper \cite{SV}, the author describes the $S_n$-module structure of $\Gamma_n(\mathfrak{N}_4)$.
	In this subsection, we construct a class of irreducible $S_n$-submodules of $\Gamma_n(\mathfrak{N}_{2k})$, that come from the $S_n$-module decomposition of $\Gamma_n(\mathfrak{N}_4)$ given in \cite{SV}. 
	
	First, we define the following polynomials.
	\begin{align*}
		&g_1^{(2j-1)} = \sum_{\sigma \in S_{2j-3}}(-1)^{\sigma} [x_{\sigma(1)}, x_{\sigma(2)}]\cdots [x_{\sigma(2j-5)}, x_{\sigma(2j-4)}][x_{\sigma(2j-3)}, x_1, x_1]; \\
		&g_2^{(2j-1)} = \sum_{\sigma \in S_{2j-3}}(-1)^{\sigma} [x_{\sigma(1)}, x_{\sigma(2)}]\cdots [x_{\sigma(2j-5)}, x_{\sigma(2j-4)}][x_2, x_1, x_{\sigma(2j-3)}]; \\
		&g_3^{(2j-1)} = \sum_{\sigma \in S_{2j-2}}(-1)^{\sigma} [x_{\sigma(1)}, x_{\sigma(2)}]\cdots [x_{\sigma(2j-5)}, x_{\sigma(2j-4)}][x_{\sigma(2j-3)}, x_{\sigma(2j-2)}, x_1];
	\end{align*}
	
	\begin{align*}
		&g_1^{(2j)} = \sum_{\sigma \in S_{2j-2}}(-1)^{\sigma} [x_{\sigma(1)}, x_{\sigma(2)}]\cdots [x_{\sigma(2j-5)}, x_{\sigma(2j-4)}][x_{\sigma(2j-3)}, x_{\sigma(2j-2)}, x_1, x_1]; \\
		&g_2^{(2j)} = [x_1, x_2]s_{2j-2}(x_1, x_2, \dots, x_{2j-2}); \\
		&g_3^{(2j)} = \sum_{\sigma \in S_{2j-1}}(-1)^{\sigma} [x_{\sigma(1)}, x_{\sigma(2)}]\cdots [x_{\sigma(2j-5)}, x_{\sigma(2j-4)}][x_{\sigma(2j-3)}, x_{\sigma(2j-2)}, [x_{\sigma(2j-1)}, x_1]];
	\end{align*}
	
	Here, by $s_m(x_1, \dots, x_m)$ we denote the standard polynomial of degree $m$, i.e.,
	\[
	s_m(x_1, \dots, x_m) = \sum_{\sigma \in S_m} (-1)^{\sigma}x_{\sigma(1)} x_{\sigma(2)} \cdots x_{\sigma(m)}.
	\]
	
	When $m = 2j-2$ we have
	\[
	s_{2j-2}(x_1, \dots, x_{2j-2}) = \frac{1}{2^{j-1}}\sum_{\sigma \in S_{2j-2}} (-1)^{\sigma} [x_{\sigma(1)}, x_{\sigma(2)}] \cdots [x_{\sigma(2j-3)}, x_{\sigma(2j-2)}].
	\]
	
	Let $f_i^{(j)}$ denote the complete multilinearization of $g_i^{(j)}$ for $i =1, 2,3$.
	It was shown in \cite{SV} that each $f_i^{(j)}$ generates an irreducible $S_j$-submodule of $\Gamma_j(\mathfrak{N}_4)$. Hence, each $f_i^{(j)}$ generates an irreducible $S_j$-submodule of $\Gamma_j(\mathfrak{N}_p)$ for $p \geq 4$. 
	
	The first goal of this section is to prove the following proposition.
	
	\begin{proposition} \label{prop_Sn_modules}
	Let $p = 2k$. For all $0 \leq l \leq k-2$ and for $i = 1,2,3$, the linearization of $g_i^{(j)}[x_1, x_2]^{l}$ generates an irreducible $S_{j+2l}$-submodule of $\Gamma_{j+2l}(\mathfrak{N}_{2k})$. The submodule generated by $g_1^{(j)}[x_1, x_2]^l$ corresponds to the partition $(l+3,l+1, 1^{j-4})$, the submodule generated by $g_2^{(j)}[x_1, x_2]^l$ corresponds to the partition $(l+2,l+2, 1^{j-4})$, and the submodule generated by $g_3^{(j)}[x_1, x_2]^l$ corresponds to the partition $(l+2,l+1, 1^{j-3})$.
	\end{proposition}

\begin{proof}
	First we show that $g_i^{(2j-1)}[x_1, x_2]^{k-2}$ and $g_i^{(2j)}[x_1, x_2]^{k-2}$ 
	are not identities for $E\otimes E_2 \otimes \underbrace{E_2\otimes \cdots \otimes E_2}_{k-2}$ for $i = 1,2,3$.
	
	Let us set 
	\[
		x_1 = (e_1\otimes g_1 + e_2 \otimes g_2 + e_3 \otimes 1) \otimes \underbrace{1\otimes \cdots \otimes 1}_{k-2} + 1\otimes 1 \otimes (\sum_{i = 1}^{k-2} 1\otimes \dots \otimes 1 \otimes g_1 \otimes 1 \otimes \cdots \otimes 1),
		\]		
		where in the second summand $g_1$ appears at position $i$ in the tensor product. Similarly,
		\[
		x_2 = e_4\otimes 1 \otimes \underbrace{1\otimes \cdots \otimes 1}_{k -2} + 1\otimes 1 \otimes (\sum_{i = 1}^{k-2} 1\otimes \dots \otimes 1 \otimes g_2 \otimes 1 \otimes \cdots \otimes 1),
		\]
		where in the second summand $g_2$ appears at position $i$ in the tensor product. In addition,
		\begin{align*}
			&x_i = e_{i+2}\otimes 1\otimes \underbrace{1\otimes \cdots \otimes 1}_{k -2} \text{ for } i = 3, \dots, 2j-3;
		\end{align*}
	
	Then, by direct computation we check that for an appropriate value of $r$,
	\[
	g_1^{(2j-1)}[x_1, x_2]^{k-2} = 2^re_1e_2\cdots e_{2j-1}\otimes \underbrace{g_1g_2\otimes \cdots \otimes g_1g_2}_{k-1} \neq 0.
	\]
 Hence, $g_1^{(2j-1)}[x_1, x_2]^{k-2}$ is not an identity for $E\otimes\underbrace{E_2\otimes \cdots \otimes E_2}_{k-1}$.
 
 We use the same evaluation to show that $g_1^{(2j)}[x_1, x_2]^{k-2}$ and $g_3^{(2j-1)}[x_1, x_2]^{k-2}$ are not identities for $E\otimes\underbrace{E_2\otimes \cdots \otimes E_2}_{k-1}$ either.
	
	To check that $g_2^{(2j)}$ is not an identity for $E\otimes\underbrace{E_2\otimes \cdots \otimes E_2}_{k-1}$, we use the evaluation
	\begin{align*}
		&x_1 = e_1 \otimes  \underbrace{1\otimes \cdots \otimes 1}_{k -1} + 1 \otimes (\sum_{i = 1}^{k-1} 1\otimes \dots \otimes 1 \otimes g_1 \otimes 1 \otimes \cdots \otimes 1); \\
		&x_2 = e_2 \otimes  \underbrace{1\otimes \cdots \otimes 1}_{k -1} + 1 \otimes (\sum_{i = 1}^{k-1} 1\otimes \dots \otimes 1 \otimes g_2 \otimes 1 \otimes \cdots \otimes 1); \\
		&x_i = e_i \otimes  \underbrace{1\otimes \cdots \otimes 1}_{k -1} \text{ for } i \geq 3.
	\end{align*}

To check that $g_2^{(2j-1)}[x_1,x_2]^{k-2}$ is not an identity for $E\otimes\underbrace{E_2\otimes \cdots \otimes E_2}_{k-1}$, we use the evaluation
\begin{align*}
	&x_1 = (e_1 \otimes g_1 + e_3 \otimes 1) \otimes \underbrace{1\otimes \cdots \otimes 1}_{k -2} + 1 \otimes 1 \otimes (\sum_{i = 1}^{k-2} 1\otimes \dots \otimes 1 \otimes g_1 \otimes 1 \otimes \cdots \otimes 1); \\
	&x_2 = (e_2\otimes 1 + 1 \otimes g_2) \otimes  \underbrace{1\otimes \cdots \otimes 1}_{k -2} + 1 \otimes 1 \otimes (\sum_{i = 1}^{k-2} 1\otimes \dots \otimes 1 \otimes g_2 \otimes 1 \otimes \cdots \otimes 1); \\
	&x_i = e_{i+1} \otimes 1 \otimes \underbrace{1\otimes \cdots \otimes 1}_{k -2} \text{ for } i \geq 3.
\end{align*}

Finally, to check that $g_3^{(2j)}[x_1,x_2]^{k-2}$ is not an identity for $E\otimes\underbrace{E_2\otimes \cdots \otimes E_2}_{k-1}$, we use the evaluation
\begin{align*}
	&x_1 = (e_1 \otimes g_1 + e_3 \otimes 1) \otimes \underbrace{1\otimes \cdots \otimes 1}_{k -2} + 1 \otimes 1 \otimes (\sum_{i = 1}^{k-2} 1\otimes \dots \otimes 1 \otimes g_1 \otimes 1 \otimes \cdots \otimes 1); \\
	&x_2 = (e_2\otimes 1) \otimes  \underbrace{1\otimes \cdots \otimes 1}_{k -2} + 1 \otimes 1 \otimes (\sum_{i = 1}^{k-2} 1\otimes \dots \otimes 1 \otimes g_2 \otimes 1 \otimes \cdots \otimes 1); \\
	&x_i = e_{i+1} \otimes 1 \otimes \underbrace{1\otimes \cdots \otimes 1}_{k -2} \text{ for } i \geq 3.
\end{align*}

In this way we proved that $g_i^{(j)}[x_1, x_2]^{l}$ for $i = 1,2,3$ and for all $0 \leq l \leq k-2$ are not identities for $E\otimes \underbrace{E_2\otimes \cdots \otimes E_2}_{k-1}$. Therefore, by Corollary \ref{coro_tensorByE2}, the linearizations of $g_i^{(j)}[x_1, x_2]^{l}$ are not elements of $I_{2k+1}$ and thus they are nonzero elements in $\Gamma_{j+2l}(\mathfrak{N}_{2k})$. 

The polynomial $g_1^{(j)}[x_1, x_2]^{l}$ belongs to $K\left\langle x_1, \dots, x_m \right\rangle$ for any $m \geq j+2l$. The group $\GL(m)$ acts naturally on $K\left\langle x_1, \dots, x_m \right\rangle$ and $g_1^{(j)}[x_1, x_2]^{l}$ is a highest weight vector for this action with highest weight $(l+3, l+1, 1^{j-4})$. Hence, from \cite{D2}, Chapter 12 (see also the beginning of \cite{DiD}) it follows that the linearization of $g_1^{(j)}[x_1, x_2]^{l}$ generates an irreducible $S_{j+2l}$-submodule of $\Gamma_{j+2l}(\mathfrak{N}_{2k})$ corresponding to the partition $(l+3, l+1, 1^{j-4})$.

The proof for the polynomials $g_i^{(j)}[x_1, x_2]^{l}$ with $i=2,3$ is analogous.
\end{proof}	
	
	\begin{remark}
		Notice that when $n \leq 2k$, then $\Gamma_{n}(\mathfrak{N}_{2k}) = \Gamma_{n}$. Similarly, when $n \leq  2k+1$, $\Gamma_n(\mathfrak{N}_{2k+1}) = \Gamma_n$. 
	\end{remark}
	

Proposition \ref{prop_Sn_modules} gives the following list of partitions such that the corresponding irreducible $S_n$-module appears in the decomposition of $\Gamma_n(\mathfrak{N}_{2k})$: 
	\[
	(l+3,l+1, 1^{n-2l-4}), (l+2,l+2, 1^{n-2l-4}), (l+2,l+1, 1^{n-2l-3}), 
	\]
	where $1 \leq l \leq k-2$. (For $l=0$ the corresponding modules appear already in Theorem \ref{thm_DiD}, that is why we exclude the case $l=0$.)

Then next goal is, using Proposition \ref{prop_Sn_modules}, to obtain new lower bounds for the dimensions of $\Gamma_n(\mathfrak{N}_{2k})$, $\Gamma_n(\mathfrak{N}_{2k+1})$, $P_n(\mathfrak{N}_{2k})$, and $P_n(\mathfrak{N}_{2k+1})$. 
	
	Let us denote by $M_{i,l}^{(j+2l)}$ the irreducible $S_{j+2l}$-module generated by the linearization of $g_i^{(j)}[x_1, x_2]^l$. 
	In other words, if we set $n = j+2l$, then $M_{i,l}^{(n)}$ is the irreducible $S_{n}$-module generated by the linearization of $g_i^{(n-2l)}[x_1, x_2]^l$. 
	Using the hook formula, we obtain the following expressions.
	
	\begin{align*}
		&	\operatorname{dim} M_{1, l}^{(n)} = \frac{3n!}{(n-2l-4)!(n-l)(n-l-3)(l+3)!l!}\\
		&	\operatorname{dim} M_{2, l}^{(n)} = \frac{n!}{(n-2l-4)!(n-l-1)(n-l-2)(l+2)!(l+1)!}\\
		&	\operatorname{dim} M_{3, l}^{(n)} = \frac{2n!}{(n-2l-3)!(n-l)(n-l-2)(l+2)!l!}\\
	\end{align*}	
	
	Therefore, for a fixed $l$ the dimensions of $M_{1, l}^{(n)}$ and $M_{2, l}^{(n)}$ are polynomials of $n$ of degree $2l+2$ and the dimension of $M_{3, l}^{(n)}$ is a polynomial of $n$ of degree $2l+1$. 

	The above results imply the following lower bound for $ \gamma_n(\mathfrak{N}_{2k})$ and $\gamma_n(\mathfrak{N}_{2k+1})$.
	\begin{corollary}\label{coro_dim_Gamma}
		For $n \geq 2k$,
		\begin{align*}
			&\gamma_{n}(\mathfrak{N}_{2k+1}) \geq \gamma_{n}(\mathfrak{N}_{2k}) \geq p_k(n),
	\end{align*}
		where, $p_k(n)$ is a polynomial of $n$ with rational coefficients of degree $2k-2$. In particular,
		\[
		p_k(n) = \sum_{l=0}^{k-2} (\dim M_{1, l}^{(n)} + \dim M_{2, l}^{(n)} + \dim M_{3, l}^{(n)}) + \varepsilon(n), 
		\]	
where $\varepsilon(n) = 1$ if $n$ is even and $\varepsilon(n) = 0$ otherwise.
		The leading coefficient of $p_k(n)$ (i.e., the coefficient in front of $n^{2k-2}$) is equal to
	\[
\frac{C_k}{(2k-2)!},	
		\]
		where $C_k = \frac{1}{k+1}\binom{2k}{k}$ denotes the $k$-th Catalan number.
	\end{corollary}
\begin{proof}
	The polynomial $p_k(n)$ is obtained by adding the dimensions of the $S_n$-modules that appear in the decomposition of $\Gamma_n(\mathfrak{N}_{2k})$. First, we will check that $p_k(n)$ is well-defined.
	
The polynomials $g_1^{(2j)}$ and $g_2^{(2j)}$ are well-defined for $j \geq 2$, and the polynomials $g_1^{(2j-1)}$ and $g_2^{(2j-1)}$ are well-defined for $j \geq 3$. 
Hence, when $n \geq 2k$, the polynomials $g_i^{(n-2l)}[x_1, x_2]^l$ (and thus also the polynomials $M_{i,l}^{(n)}$) are well-defined for all $0 \leq l \leq k-2$.

When $n$ is even, there is one more $S_n$-module that appears in the $S_n$-decomposition of $\Gamma_n(\mathfrak{N}_4)$, namely the one-dimensional module generated by $s_{n}(x_1, \dots, x_{n})$ (see \cite{SV}). This explains the term $\varepsilon(n)$ in the formula for $p_k(n)$.

Finally, for the leading coefficient of $p_k(n)$ we obtain the following:
		\[
		\frac{2(2k-1)}{(k+1)!(k-1)!} = \frac{1}{(2k-2)!(k+1)}\binom{2k}{k} = \frac{C_k}{(2k-2)!}.	
		\]
\end{proof}	


Notice that when $k=2$ the bound for $\gamma_n(\mathfrak{N}_{2k})$ and $\gamma_n(\mathfrak{N}_{2k+1})$ obtained in Corollary \ref{coro_dim_Gamma} is equal to the respective bound from Corollary \ref{coro_dim_FromE} (i) (which is natural because in both cases we just consider $\gamma_n(E\otimes E_2))$. Furthermore, when $k \geq 4$, we have that the polynomials $g_i^{(j)}[x_1, x_2]^{k-2}$ correspond to partitions that do not appear in the decomposition of $E \otimes E_{2k-2}$. Therefore, for $k \geq 4$ we can add the lower bounds found in Corollary \ref{coro_dim_FromE} (i) and Corollary \ref{coro_dim_Gamma} to obtain a better lower bound for $\gamma_n(\mathfrak{N}_{2k})$ and $\gamma_n(\mathfrak{N}_{2k+1})$. When $k=3$, the polynomial $g_2^{(j)}[x_1, x_2]^{k-2}$ corresponds to a partition that does not appear in the decomposition of $E \otimes E_{2k-2}$. Therefore, in this case we can add the leading coefficient of the polynomial $\dim M_{2,k-2}^{(n)}$ to the bound found in Corollary \ref{coro_dim_FromE} (i). These remarks lead to the following statement.

	\begin{corollary} \label{coro_final_bound_Gamma}
	Let $n \geq 2k$. Then,
	
	\[
	\gamma_n(\mathfrak{N}_{2k+1}) \geq \gamma_n(\mathfrak{N}_{2k}) \geq p_k(n),
	\]
	
	where $p_k(n)$ is a polynomial in $n$ with rational coefficients of degree $2k-2$.
	
	When $k \geq 4$ the leading coefficient of $p_k(n)$ is equal to
	\begin{align*}
		\frac{2^{2k-3} + C_k}{(2k-2)!},
	\end{align*}
	where again $C_k$ is the $k$-th Catalan number.
	
	When $k = 3$, the leading coefficient of $p_k(n)$ is equal to
		\begin{align*}
		\frac{2^{2k-3} + C_{k-1}}{(2k-2)!}.
	\end{align*}
	
	For $n \leq 2k$, it holds that $\Gamma_n(\mathfrak{N}_{2k+1}) = \Gamma_n(\mathfrak{N}_{2k}) = \Gamma_n$.
	
	Furthermore, for $2k+1 \leq n \leq 4k$ we have that $\gamma_n(\mathfrak{N}_{2k+1})$ is strictly bigger than $\gamma_n(\mathfrak{N}_{2k})$.
\end{corollary}	

Next, using Corollary \ref{coro_final_bound_Gamma} and the methods from \cite{DiD}, we obtain the following final lower bound for the codimensions $c_n(\mathfrak{N}_{2k})$ and $c_n(\mathfrak{N}_{2k+1})$.

\begin{corollary} \label{coro_final_bound_codim}
	Let $n \geq 2k+1$. Then,
	\[
	c_n(\mathfrak{N}_{2k+1}) \geq c_n(\mathfrak{N}_{2k}) \geq 2^n R(n) + S(n),
	\]
	where $R(n), S(n) \in \QQ[n]$, $\deg R(n) = 2k-2$ and 
	$\deg S(n) \leq 2k-1$.
	When $k \geq 4$, the leading coefficient of $R(n)$ is
	\[
	\frac{2^{2k-3}+ C_k}{2^{2k-2}(2k-2)!},
	\]
	where again $C_k$ is the $k$-th Catalan number.
	
	When $k = 3$, the leading coefficient of $R(n)$ is
	\[
	\frac{2^{2k-3}+ C_{k-1}}{2^{2k-2}(2k-2)!}.
	\]

	For $n \leq 2k$, we have that $c_n(\mathfrak{N}_{2k+1}) = c_n(\mathfrak{N}_{2k}) = \dim P_n = n!$.
	
	Furthermore, for $2k+1 \leq n \leq 4k$ we have that $c_n(\mathfrak{N}_{2k+1})$ is strictly bigger than $c_n(\mathfrak{N}_{2k})$.
\end{corollary}	

\begin{proof}
	The proof is similar to the proof of Theorem 3.5 (ii) from \cite{DiD}.
	By Corollary \ref{coro_final_bound_Gamma} we have that for $n \geq 2k$,
		\[
		\gamma_n(\mathfrak{N}_{2k}) \geq p_k(n),
		\]
		where $p_k(n) \in \QQ[n]$, $\deg p_k(n) = 2k-2$ and the leading coefficient of $p_k(n)$ is denoted by $\alpha_n$. When $k \geq 4$, we have $\alpha_n = \frac{2^{2k-3} + C_k}{(2k-2)!}$ and when $k = 3$, then $\alpha_n = \frac{2^{2k-3} + C_{k-1}}{(2k-2)!}$.
		
For $n \leq 2k-1$, we set
		\[
		\gamma_n(\mathfrak{N}_{2k}) = p_k(n) + (\operatorname{dim} \Gamma_n - p_k(n)),
		\]
		where $p_k(n)$ is the same polynomial as above.
We set $\nu_n = \operatorname{dim} \Gamma_n - p_k(n)$ and for each $n$ we have that $\nu_n \in \QQ$. 

The polynomials $\binom{n+m}{m}$ for $m = 0,1,2 \dots$ form a basis for $\QQ[n]$. Therefore, for some coefficients $a_0, \dots, a_{2k-2} \in \QQ$ we have that
\[
p_k(n) = \sum_{m = 0}^{2k-2} a_m \binom{n+m}{m}.
\]

Notice that the leading coefficient of $\binom{n+m}{m}$ for each $m$ is equal to $\frac{1}{m!}$. Hence, for the coefficient $a_{2k-2}$ we obtain the following expression
\[
a_{2k-2} = \alpha_n(2k-2)!.
\]

We define the function $p'_k(n)$ in the following way:
\begin{align*}
	p'_k(n) = \begin{cases} 
		p_k(n) & \text{ if } n \geq 2k \\
		p_k(n) + \nu_n & \text{ otherwise} .
	\end{cases}	 
\end{align*}

Therefore, for all $n$ it holds that $\gamma_n(\mathfrak{N}_{2k}) \geq p'_k(n)$.

The next step is to consider the series
\begin{align*}
p_k(t) = \sum_{n \geq 0} p'_k(n) t^n = \sum_{n \geq 0} \sum_{m = 0}^{2k-2} a_m \binom{n+m}{m}t^n + \sum_{n = 0}^{2k-1} \nu_n t^n.
\end{align*}

Now we use the property that
\begin{align} \label{eq_binom}
	\sum_{n\geq 0} \binom{n+m}{m}t^n = \frac{1}{(1-t)^{m+1}}.
\end{align}	 
Hence,
\[
p_k(t) = \sum_{m=0}^{2k-2}\frac{a_m}{(1-t)^{m+1}} + \theta_k(t),
\]
where $\theta_k(t) = \sum_{n = 0}^{2k-1} \nu_n t^n$ and $\theta_k(t)$ is a polynomial in $t$ of degree at most $2k-1$.

Next, we define the series
\[
q_k(t) = \sum_{n \geq 0} q_k(n) t^n,
\]
where 
\[
q_k(n) = \sum_{l = 0}^n \binom{n}{l}p'_k(n).
\]

By Proposition \ref{prop_codim_propcodim}, it follows that for each $n$
\[
c_n(\mathfrak{N}_{2k}) \geq q_k(n).
\]

To compute the numbers $q_k(n)$ we notice that
\[
q_k(t) = \frac{1}{1-t}p_k\left(\frac{t}{1-t}\right).
\]
Therefore,
\begin{align*}
	q_k(t) = \sum_{m=0}^{2k-2} \frac{a_m (1-t)^{m}}{(1-2t)^{m+1}} + \frac{1}{1-t}\theta_k\left(\frac{t}{1-t}\right).
\end{align*}

First, we simplify the expression
\begin{align*}
\frac{(1-t)^m}{(1-2t)^{m+1}} = \frac{(1-(1-2t))^m}{2^m(1-2t)^{m+1}} = \frac{1}{2^m (1-2t)^{m+1}} + \rho_m\left(\frac{1}{1-2t}\right),
\end{align*}
where $\rho_m(t) \in \QQ[t]$ and $\deg \rho_m(t) \leq m$.

We also set 
\[
\frac{1}{1-t}\theta_k\left(\frac{t}{1-t}\right) = \tau_k\left(\frac{1}{1-t}\right),
\]
where $\tau_k(t) \in \QQ[t]$ and $\deg \tau_k(t) \leq 2k$. We also notice that $\tau_k(t)$ has no free term.

Therefore,
\begin{align*}
	q_k(t) = \frac{a_{2k-2}}{2^{2k-2}(1-2t)^{2k-1}} + \sum_{m = 0}^{2k-3}\frac{a_m'}{(1-2t)^{m+1}} + \tau_k\left(\frac{1}{1-t}\right).
\end{align*}

Here, the coefficient $a_m'$ comes from $a_m$ and $\rho_m\left(\frac{1}{1-2t}\right)$.

We also set $\tau_k(t) = \sum_{m = 1}^{2k} \tau_{km}t^m$ for some coefficients $\tau_{km} \in \QQ$. Hence,
\[
\tau_k\left(\frac{1}{1-t}\right) = \sum_{m=1}^{2k} \frac{\tau_{km}}{(1-t)^m} = \sum_{m=0}^{2k-1} \frac{\tau_{k,m+1}}{(1-t)^{m+1}}.
\]

Now, we use again property (\ref{eq_binom}) and obtain
\begin{align*}
&q_k(t) = \sum_{n \geq 0} \left(\frac{a_{2k-2}}{2^{2k-2}} \binom{n+2k-2}{2k-2} + \sum_{m = 0}^{2k-3} a'_m \binom{n+m}{m}\right) 2^n t^n + \\
& \sum_{m = 0}^{2k-1} \tau_{k,m+1}\binom{n+m}{m}t^n.
\end{align*}

The last equation implies that
\[
q_k(n) = \left(\frac{a_{2k-2}}{2^{2k-2}} \binom{n+2k-2}{2k-2} + \sum_{m = 0}^{2k-3} a'_m \binom{n+m}{m}\right) 2^n + \sum_{m = 0}^{2k-1} \tau_{k,m+1}\binom{n+m}{m}.
\]

Hence, we have that $c_n(\mathfrak{N}_{2k+1})\geq c_n(\mathfrak{N}_{2k}) \geq q_k(n) = R(n)2^n + S(n)$, where $\deg R(n) = 2k-2$ and $\deg S(n) \leq 2k-1$ and the leading coefficient of $R(n)$ is equal to
\[
\frac{a_{2k-2}}{2^{2k-2}(2k-2)!}.
\]

When $k \geq 4$, we obtain
\[
\frac{a_{2k-2}}{2^{2k-2}(2k-2)!} = 	\frac{2^{2k-3}+ C_k}{2^{2k-2}(2k-2)!}.
\]

Similarly, when $k = 3$, we obtain
\[
\frac{a_{2k-2}}{2^{2k-2}(2k-2)!} = 	\frac{2^{2k-3}+ C_{k-1}}{2^{2k-2}(2k-2)!}.
\]
\end{proof}

\section*{Acknowledgements}
I am grateful to Vesselin Drensky for the fruitful discussions that we had during my work on this paper.

\end{document}